\documentclass{amsart}
\newcommand*\patchAmsMathEnvironmentForLineno[1]{%
  \expandafter\let\csname old#1\expandafter\endcsname\csname #1\endcsname
  \expandafter\let\csname oldend#1\expandafter\endcsname\csname end#1\endcsname
  \renewenvironment{#1}%
     {\linenomath\csname old#1\endcsname}%
     {\csname oldend#1\endcsname\endlinenomath}}%
\newcommand*\patchBothAmsMathEnvironmentsForLineno[1]{%
  \patchAmsMathEnvironmentForLineno{#1}%
  \patchAmsMathEnvironmentForLineno{#1*}}%
\AtBeginDocument{%
\patchBothAmsMathEnvironmentsForLineno{equation}%
\patchBothAmsMathEnvironmentsForLineno{align}%
\patchBothAmsMathEnvironmentsForLineno{flalign}%
\patchBothAmsMathEnvironmentsForLineno{alignat}%
\patchBothAmsMathEnvironmentsForLineno{gather}%
\patchBothAmsMathEnvironmentsForLineno{multline}%
}

\usepackage{amsfonts, amssymb, amsmath, latexsym, enumerate, array,caption,subcaption}
\usepackage{amscd}     
\usepackage[all]{xy}
\usepackage[pdftex]{graphicx}

\usepackage[latin1]{inputenc}
\usepackage{hyperref,cite,mdwlist,mathrsfs,color}

\usepackage[left=3.3cm,top=2.7cm,right=3.3cm]{geometry}

\newtheorem{thm}{Theorem}[section] \newtheorem{lem}[thm]{Lemma}
\newtheorem{coro}[thm]{Corollary} 
 \newtheorem{prop}[thm]{Proposition}

 \newtheorem{defi}[thm]{Definition}
\newtheorem{rem*}{Remark}
\newtheorem{rems*}{Remarks}
\newtheorem{problem*}{Problem}
\newtheorem{conj*}{Conjecture}

\newcommand{\sO}{\mathscr{O}}

\newcommand{\cI}{\mathcal{I}}

\DeclareMathOperator{\HH}{H} \DeclareMathOperator{\hh}{h}

\newcommand{\Z}{\mathbb Z} \newcommand{\C}{\mathbb C}

 \newcommand{\p}{\mathbb P}

\SelectTips{cm}{} \newdir{ >}{{}*!/-5pt/@{>}}
\numberwithin{equation}{section}

\title[Singular hypersurfaces - Lefschetz properties]{Singular hypersurfaces characterizing the Lefschetz properties}

\author[R. Di Gennaro, G. Ilardi and J. Vall\`es]{Roberta Di Gennaro, Giovanna Ilardi and Jean Vall\`es}

\thanks{Third author partially supported by ANR-09-JCJC-0097-0 INTERLOW,  ANR GEOLMI,  the project \lq\lq F.A.R.O. 2010:
 Algebre di Hopf, differenziali e di vertice in geometria, topologia e
 teorie di campo classiche e quantistiche\rq\rq\ and by the program \lq\lq Scambi internazionali tra l'Universit\`{a} degli
Studi di Napoli Federico II ed Istituti di ricerca stranieri per la mobilit\`{a} di breve durata
di docenti, studiosi e ricercatori\rq\rq.}

\begin{document}
\maketitle

\begin{abstract} In a recent paper \cite{MMO} Miro-Roig, Mezzetti and Ottaviani highlight the link between rational varieties
satisfying a Laplace equation and artinian ideals failing the Weak Lefschetz Property.
Continuing their work we extend this link to the more general
situation of artinian ideals failing the Strong Lefschetz Property.
We characterize the failure of the SLP (which includes WLP) by the existence of
special singular hypersurfaces (cones for WLP).  This characterization allows us to solve three problems posed in
\cite{MN} and to give new examples of ideals failing the SLP.
Finally, line arrangements are related to  artinian ideals and the unstability of the associated  derivation bundle
 is linked  to the failure of the SLP. Moreover
 we reformulate the so-called Terao's conjecture for free  line arrangements in terms of artinian ideals failing the
SLP.
\end{abstract}


\section{Introduction}

The  tangent space  to an integral projective variety $X \subset \mathbb P^N$ of dimension $n$ in a smooth point $P$, named $T_P X$,
 is always of dimension  $n$. It is no longer true for the osculating spaces. For instance, as it was pointed out by Togliatti in \cite{To}, the
 osculating space  $T^2_PX$, in a  general point $P$, of
 the rational surface $X$ defined by
$$ \p^2 \stackrel{\phi}\longrightarrow \p^5, \,\, (x,y,z) \mapsto (xz^2,yz^2,x^2z, y^2z, xy^2,x^2y),$$
 is of projective dimension  $4$ instead of $5$.  Indeed there is a non trivial linear relation
 between the partial derivatives of order $2$ of $\phi$  at $P$ that define $T^2_PX$. This relation is usually called a \textit{Laplace equation} of order $2$.
More generally,
we will say that $X$  satisfies a Laplace equation of order $s$ when its $s$-th osculating space $T^s_PX$ in a general point $P\in X$  is of dimension less than the expected one,  that is $\inf\{N,\binom{n+s}n-1\}$.

The study of the surfaces satisfying a Laplace equation was developed in the last century by  Togliatti \cite{To} and  Terracini  \cite{Te}. Togliatti \cite{To}
gave a complete classification of the rational surfaces embedded
  by linear systems of plane cubics and satisfying a Laplace equation of
order two.

In the paper \cite{P}, Perkinson gives a complete classification of smooth toric surfaces (Theorem 3.2) and threefolds (Theorem 3.5) embedded
  by a monomial linear system and satisfying a Laplace equation of any order.

Very recently Miro-Roig, Mezzetti and Ottaviani \cite{MMO}
have established  a nice link between rational  varieties (i.e. projections  of Veronese varieties) satisfying a Laplace equation
and artinian graded rings $A=\oplus_{0\le i\le s} A_i$ such that the multiplication by a general linear form
has not  maximal rank in  a degree $i$. On the contrary, when the rank of the multiplication map is maximal in any degree, the ring is said to have the
\textit{Weak Lefschetz Property} (briefly WLP).
The same type of problems  arises when we consider  the multiplication by powers $L^k$ ($k\ge 1$) of a general linear form $L$.
Indeed, if the rank of the multiplication map by $L^k$ is maximal for any $k$ and  any degree, the ring is said to have the
\textit{Strong Lefschetz Property} (briefly SLP).
\\ These properties are so called after Stanley's seminal work: the Hard Lefschetz theorem is
used to prove  that the ring $\frac{\C[x_0,\ldots,x_n]}{(x_0^{d_0},\ldots,x_n^{d_n})}$ has the  SLP  \cite[Theorem 2.4]{St}.
From this example one can ask if the artinian complete intersection rings  have the WLP.
Actually  $\frac{\C[x,y,z]}{(F_0,F_1,F_2)}$ has the WLP (first proved  in \cite{HMNW} and then also in \cite{BK})
but   it is still not known  for more than three variables. Many other questions derive from this first example.
\\
For more details about  known results and some open problems we refer to \cite{MN}.
\par

 Let $I=(F_1,\ldots, F_r)$ be an artinian ideal generated by the $r$ forms $F_1,\ldots, F_r$, all of the same degree $d$, and $Syz(I)$ be the \textit{syzygy bundle} associated to $I$ and defined
  in the following way:
 $$  \begin{CD}
  0@>>> Syz(I)(d) @>>>  \sO_{\p^{n}}^{r} @>(F_1, \ldots, F_r)>>
  \sO_{\p^{n}}(d) @>>> 0.
 \end{CD}$$
 For shortness we will denote $K=Syz(I)(d) $ and, forgetting the twist by $d$, in all the rest of this text we call it the syzygy bundle.
 As in   \cite{HMNW},  many papers about the Lefschetz properties  involve the \textit{syzygy bundle}.
  Indeed, in \cite[Proposition 2.1]{BK}, Brenner and Kaid  prove that the graded piece of degree $d+i$ of the artinian ring $A=\frac{\C[x_0,\ldots,x_n]}{(F_0,\ldots ,F_r)}$ is  $\HH^1(K(i))$.
 In
  [\cite{MMO}, thm. 3.2] the authors characterize  the failure of the WLP (in degree $d-1$, i.e. for the map $A_{d-1}\rightarrow A_d$) when $r\le \hh^0(\sO_{L}(d))$
 by the non injectivity of the restricted map
 $$  \begin{CD}
 \HH^0(  \sO_{L})^{r} @>(F_1,\ldots,F_r)>>
  \HH^0(\sO_{L}(d)),
 \end{CD}$$
 on a general hyperplane $L$.
%
%
%
%

Let us say,  in few words, what we are doing in this paper and how  it is organized. First of all we recall some definitions,
basic facts and we propose a conjecture (Section \ref{s1}). In Section \ref{s2} we
 extend to the SLP  the characterization of  failure of the WLP  given in \cite{MMO}.  Then we translate the failure of the WLP and SLP in terms of existence of special singular hypersurfaces  (Section \ref{s3}).
It allows us to  give an answer to three unsolved questions in \cite{MN}. In Section \ref{s4} we
construct examples of artinian rings
failing the WLP and the SLP by producing the appropriate singular hypersurfaces. In the last section we relate the problem of SLP at the range 2 to the topic of line arrangements  (Section \ref{s5}).

Let us now give more details about  the different sections of this paper.
In  Section \ref{s2}, more precisely in Theorem \ref{p1},
 we characterize the failure of the SLP by the non maximality of the induced map on  sections
$$  \begin{CD}
\HH^0(  \sO_{L^k}(i))^{r} @>(F_1,\ldots,F_r)>>
 \HH^0(\sO_{L^k}(i+d)).
\end{CD}$$
The geometric consequences of this link are explained in  Section \ref{s3} (see Theorem \ref{th1bis}). The non injectivity is translated in terms of  the number of Laplace equations
and the non surjectivity is related, via apolarity, to the existence of special singular hypersurfaces.
Then we give Propositions \ref{pr54-1}, \ref{pr54-2} and \ref{pr54-3} that solve three problems posed in  \cite[Problem 5.4 and Conjecture 5.13]{MN}.

In Section \ref{s4} we produce many  examples of ideals (monomial and  non monomial) that fail the WLP and the SLP. The failure of the WLP is studied for monomial ideals generated in degree $4$ on $\p^2$ (Theorem \ref{th3}),
 in degree $5$ on $\p^2$ (Proposition \ref{th4}),  in degree $4$ on $\p^3$ (Proposition \ref{d4m}); the failure of the SLP
is studied for monomial  ideals  generated in degree $4$  (Proposition \ref{d4mslp}); finally, we propose a method to produce non monomial ideals that fail the SLP at any range (Proposition \ref{nmslp}).

In the last section  Lefschetz  properties and line arrangements are linked. The theory of line arrangements, more generally of hyperplane arrangements, is an old and deep subject that concerns
combinatorics, topology and algebraic geometry. One can say that it began with Jakob Steiner (in the first volume of Crelles's journal, 1826) who determined in how many regions a real plane
is divided by a finite number of lines. It is relevant also with Sylvester-Gallai's amazing problem.
Hyperplane arrangements  come back in  a modern presentation in Arnold's fundamental work \cite{A} on the cohomology ring of  $\p^n\setminus D$ (where $D$ is the union of the hyperplanes of the arrangement).
 For a large part of mathematicians working on arrangements, it culminates today with the Terao conjecture (see the last section of this paper or directly \cite{OT}).
This conjecture concerns particularly the derivation sheaf  (also called logarithmic sheaf) associated to the arrangement.  In this paper we  recall the conjecture.
In Proposition \ref{th5}
we prove that the failure of the SLP at the range 2 of some ideals is equivalent to the unstability of the associated derivation sheaves.
Thanks to the important  literature on arrangements, we find artinian ideals that  fail the SLP. For instance the Coxeter arrangement, called B3, gives an original ideal that fails the SLP at the range 2 in a non trivial way (see Proposition \ref{B3}).

We finish by a reformulation of Terao's conjecture in terms of SLP.
\section{Notations}
\noindent The ground field  is $\C$.\\
The dual $\mathrm{Hom}_{\C}(V,\C)$ of  a vector space $V$ is    denoted by $V^*$.\\
The dimension of the vector space $  \mathrm{H}^0(\sO_{\p^n}(t))$ is denoted by $r_t$ where $n$ is clearly known in the context.\\
The vector space generated by the set $E\subset \C^t$ is  $<E>$.\\
The join variety of $s$ projective varieties $X_i\subset \p^n$ is denoted  by $\mathrm{Join}(X_1,\cdots,X_s)$ (see \cite{H} for the definition of join variety).\\
The fundamental points $(1,0,\ldots, 0), (0,1,\ldots, 0), \ldots, (0,0,\ldots, 0,1)$ in $\p^n$ are denoted by $P_0, P_1, \ldots, P_n$.\\
We often write in the same way a projective hyperplane and the linear form defining it;  we use in general  the notation $L_i$ on $\p^n$ and the notation $l_i$ on $\p^2$ for  hyperplanes.\\
The ideal sheaf of a point $P$ is $\cI_P$ .
\section{Lefschetz properties}
\label{s1}
Let
$R=\C[x_0,\ldots, x_n]=\bigoplus R_t$ be the graded polynomial ring in $n+1$ variables over $\C$.  The dimension of the vector space $R_t$ is $r_t$.
\\Let
$$A=R/I= \bigoplus_{i=0}^{m}A_i$$ be a graded artinian algebra, defined by the ideal $I$. Note that $A$ is finite dimensional over $\C$.
\begin{defi}
 The artinian algebra $A$ (or the artinian ideal $I$) has the Weak Lefschetz Property (WLP) if  there exists a linear form $L$  such that the homomorphism induced by the multiplication by $L$,
$$ \times L : A_i \rightarrow A_{i+1},$$
has maximal rank (i.e. is injective or surjective) for all $i$. The artinian algebra $A$ (or the artinian ideal $I$)  has the Strong
Lefschetz Property (SLP) if  there exists a linear form $L$  such that
$$ \times L^k : A_i \rightarrow A_{i+k},$$
has maximal rank (i.e. is injective or surjective) for all $i$ and $k$.
\end{defi}
\begin{rems*}
\begin{itemize}
\item It is clear that the SLP for $k=1$ corresponds to the WLP.
\item  Actually, it can be proved that if a Lefschetz element exists, then there is an open set of such elements, so that one can call \lq\lq general linear form\rq\rq \ such an element.
\item  We will often be interested in artinian rings $A$ that fail the SLP (or WLP), i.e. when
for  any linear form $L$ there exist $i$ and $k$ such that the multiplication map
$$ \times L^k : A_i \rightarrow A_{i+k},$$
has not maximal rank. In that case we will say that $A$ (or $I$) fails the SLP at range $k$ and degree $i$. When $k=1$ we will say simply that $A$ fails the
WLP in degree $i$.
\end{itemize}
\end{rems*}
One of the main  examples  comes from  Togliatti's result (see for instance \cite{BK}, Example 3.1): the ideal $I=(x^3,y^3,z^3,xyz)$ fails the WLP in degree $2$.
There are many ways to prove it.
One of them comes from the polarity on the rational normal cubic curve. It leads to a generalization that gives one of the few known non toric examples.
\begin{prop}(\cite[Theorem 3.1]{V1})
 Let $n\ge 1$ be an integer and $l_1, \ldots, l_{2n+1}$ be non concurrent linear forms on $\p^2$. Then the ideal
$$ (l_1^{2n+1}, \ldots, l_{2n+1}^{2n+1}, \prod_{i=1}^{2n+1}l_i)$$
fails the WLP in degree $2n$.
\end{prop}
 Indeed on the general line $l$ the $2n+2$ forms of degree $2n+1$
become dependent thanks to the polarity on the rational normal curve of degree $2n+1$.
 We propose the following conjecture. For $n=1$ it is again  Togliatti's result.
\begin{conj*}
Let $l_1,\ldots,l_{2n+1} $ be non concurrent linear forms on $\p^2$ and $f$ be a form of degree $2n+1$ on $\p^2$.
Then the ideal $(l_1^{2n+1}, \ldots, l_{2n+1}^{2n+1}, f)$ fails the WLP in degree $2n$ if and only if
$f\in (l_1^{2n+1}, \ldots, l_{2n+1}^{2n+1},\prod_{i=1}^{2n+1}l_i).$
\end{conj*}
\section{Lefschetz properties and the syzygy bundle}
\label{s2}
In  \cite[Proposition 2.3]{MMO}, the failure of the WLP in degree $d-1$ is related to the restriction of the syzygy bundle  to a general hyperplane.
 Here we extend this relationship to the SLP situation at any range and in many degrees,  by using the syzygy bundle method originated in \cite{HMNW}.
\begin{thm}
\label{p1}
 Let $I=(F_1, \ldots, F_r) \subset R$ be an artinian ideal generated by   homogeneous
forms of degree $d$ and  $K$ the syzygy bundle defined by the exact sequence
$$  \begin{CD}
 0@>>> K @>>>  \sO_{\p^{n}}^{r} @>\Phi_{I}>>
 \sO_{\p^{n}}(d) @>>> 0,
\end{CD}$$
where $\Phi_{I}(a_1,\ldots,a_r)=a_1F_1+\ldots+a_rF_r.$
Let $i$ be a non-negative integer such that $ \mathrm{h}^0( K(i))=0$ and $k$ be an integer such that $k\ge 1$.
Then $I$ fails the SLP at the range  $k$ in degree $d+ i-k$  if and only if the induced homomorphism on sections
(denoted by $\mathrm{H}^0(\Phi_{I,L^k})$)
$$  \begin{CD}
  \mathrm{H}^0(  \sO_{L^k}(i))^r @>\mathrm{H}^0(\Phi_{I,L^k})>>
 \mathrm{H}^0(  \sO_{L^k}(i+d))
\end{CD}$$
has not maximal rank for a general linear form $L$.
\end{thm}
\begin{rem*}
The theorem is not true if  $ \mathrm{h}^0( K(i))\neq 0$ i.e. if there  exists a syzygy of degree $i$  among $F_1,\ldots,F_r$. In \cite{MMO} the authors consider the injectivity of the map
$\mathrm{H}^0(\Phi_{I,L^k})$ for $i=0$ and for $r\le \mathrm{h}^0(\sO_{L}(d))$. In that case, since the forms $F_j$ are the generators of $I$, we have
of course  $ \mathrm{h}^0( K)=0$.
\end{rem*}
\begin{proof}
In  \cite[Proposition 2.1]{BK} the authors proved that $A_{d+i}=\mathrm{H}^1(K(i))$ for any $i\in \Z$.
Let us consider the canonical exact sequence
$$  \begin{CD}
 0@>>> K(i-k) @> \times L^k>> K(i) @>>>
 K\otimes \sO_{L^k}(i) @>>> 0.
\end{CD}$$
We obtain a long exact sequence of cohomology
$$
0 \rightarrow
  \mathrm{H}^0(K\otimes \sO_{L^k}(i)) \rightarrow A_{d+ i-k} \stackrel{\times L^k}\longrightarrow  A_{d+i} \rightarrow
\mathrm{H}^1(K\otimes \sO_{L^k}(i))\rightarrow \mathrm{H}^2(K(i-k)) \rightarrow 0.
$$
Let us assume first that  $n> 2$. Then  we have always
$\mathrm{h}^2(K(i-k))=0$ and it gives a shorter exact sequence:
$$
0 \longrightarrow
  \mathrm{H}^0(K\otimes \sO_{L^k}(i)) \longrightarrow A_{d+ i-k} \stackrel{\times L^k}\longrightarrow  A_{d+i} \longrightarrow
\mathrm{H}^1(K\otimes \sO_{L^k}(i))\longrightarrow  0.
$$
Moreover, since $n>2$, we have also $\mathrm{h}^1(\sO_{L^k}(i)) =0$. Then by tensoring the exact sequence defining the bundle $K$
by $\sO_{L^k}(i)$ and  taking the long cohomology exact sequence, we find:
$$ 0\longrightarrow
  \mathrm{H}^0(K\otimes   \sO_{L^k}(i)) \longrightarrow \mathrm{H}^0(  \sO_{L^k}(i))^r \stackrel{\mathrm{H}^0(\Phi_{I,L^k})}\longrightarrow
 \mathrm{H}^0(  \sO_{L^k}(i+d))\longrightarrow  \mathrm{H}^1(K\otimes   \sO_{L^k}(i))
\rightarrow 0.$$
Since the kernel and cokernel of both maps, $\mathrm{H}^0(\Phi_{I,L^k}) $ and $\times L^k$ are the same, the theorem is  proved for $n>2$.

If $n=2$, let us introduce the number $t=\mathrm{h}^2(K(i-k))$. This number is equal to $t=rr_{k-i-3}-r_{k-i-d-3}$ and we have a long exact sequence:
$$ 0\rightarrow
  \mathrm{H}^0(K\otimes   \sO_{L^k}(i)) \longrightarrow A_{d+i-k}\stackrel{\times L^k}\longrightarrow  A_{d+i}\longrightarrow
 \mathrm{H}^1( K\otimes \sO_{L^k}(i))\longrightarrow \C^t
\rightarrow 0.$$
Let us consider now the long exact sequence:
 $$\begin{CD} 0@>>>\mathrm{H}^0(K\otimes   \sO_{L^k}(i)) @>>>\mathrm{H}^0(  \sO_{L^k}(i))^r @>\mathrm{H}^0(\Phi_{I,L^k})>>
 \mathrm{H}^0( \sO_{L^k}(i+d)) @>>> \\
@>>>\mathrm{H}^1(K\otimes   \sO_{L^k}(i)) @>>>\mathrm{H}^1(  \sO_{L^k}(i))^r @>>>\mathrm{H}^1( \sO_{L^k}(i+d))@>>>0.
\end{CD} $$
Since $\mathrm{h}^1(\sO_{L^k}(i)) =\mathrm{h}^2(\sO_{\p^2}(i-k))=r_{k-i-3}$ (and  $\mathrm{h}^1(\sO_{L^k}(i+d)) =\mathrm{h}^2(\sO_{\p^2}(i+d-k))=r_{k-i-d-3}$), it remains a shorter exact sequence
 $$\begin{CD} 0@>>>\mathrm{H}^0(K\otimes   \sO_{L^k}(i)) @>>>\mathrm{H}^0(  \sO_{L^k}(i))^r @>\mathrm{H}^0(\Phi_{I,L^k})>>
 \mathrm{H}^0( \sO_{L^k}(i+d))   \\
@>>>\mathrm{H}^1(K\otimes\sO_{L^k}(i))@>>>\C^t@>>>0.
\end{CD} $$
As before, since the kernel and cokernel of both maps are the same, the theorem is  proved.
\end{proof}
Let us  introduce  the numbers
$N(r,i,k,d):=r(r_i-r_{i-k})- (r_{d+i}-r_{d+i-k})$,
$$N^{+}=\mathrm{sup}(0,N(r,i,k,d)) \,\, \mathrm{and}\,\, N^{-}=\mathrm{sup}(0,-N(r,i,k,d)). $$
The following corollary  is a didactic reformulation of the theorem above.
\begin{coro}
Assume that  there is no syzygy of degree $i$  among the $F_j$'s.  Then $I$ fails the SLP at the range $k$ in degree $d+i-k$
if and only if one of the two following equivalent conditions occurs:
\begin{itemize}
\item $\mathrm{h}^0(K\otimes   \sO_{L^k}(i))=\dim_{\mathbb C} (\mathrm{ker}(\mathrm{H}^0(\Phi_{I,L^k})))> N^{+}$,
\item $\dim_{\mathbb C} (\mathrm{coker}(\mathrm{H}^0(\Phi_{I,L^k})))> N^{-}$.
\end{itemize}
\end{coro}
In the next section we  translate this corollary in  geometric terms.
\section{Syzygy bundle and Veronese variety}
\label{s3}
We recall that the $s$-th osculating space $T_P^{s}(X)$ to a $n$-dimensional complex projective  variety $X\subset \p^N$ at $P$ is the subspace of $ \p^N$ spanned
by $P$ and by all the derivative points of degree less  than or equal to $s$ of a local parametrization of $X$, evaluated at $P$. Of course, for $s=1$ we get the tangent space $T_P(X)$.
A $n$-dimensional variety $X\subset \p^N$ whose $s$-th osculating space at a general point has dimension $\mathrm{inf}(\binom{n+s}{n}-1,N)-\delta$ is said to satisfy $\delta$ independent Laplace equations of order $s$. We will say, for shortness, that the \textit{number} of Laplace equations is $\delta$.

\begin{rem*}
If $N < \binom{n+s}{n}-1$, then there are always $\binom{n+s}{n}-1-N$ linear relations between the partial derivatives. These relations are \lq\lq trivial\rq\rq\ Laplace equations of order s. We will not consider them in the following, so when we write \lq\lq there is a Laplace equation of order $s$\rq\rq\ we understand \lq\lq a non-trivial Laplace equation of order $s$\rq\rq.
\end{rem*}
Let us briefly explain now the link with projections of $v_t(\p^n)$.

Let $R_1$ be a complex vector space of linear forms of  dimension $n+1$ such that $\mathrm{H}^0\sO_{\p^n}(1)=R_1$.
We consider the Veronese embedding:
$$
\begin{array}{llll}
v_{t} : & \p (R_1^*)&\hookrightarrow &\p (R_{t}^*)\\
 &  [L]&  \mapsto &  [L^{t}].
\end{array}
$$
The image  $v_t(\p^n)$ is called the Veronese $n$-fold of order $t$.
At the point $[L^t]\in v_t(\p^n)$, the $s$-th osculating space, $1\le s\le t-1$,
is the space of degree $d$ forms  possessing a  factorization
$L^{t-s}G$ where $G$ is a form of degree $s$ \cite[Theorem 1.3]{I2}.  It is identified with
$\p (R_s^*)$.

Let us think about the projective duality in terms of derivations  (it is in fact the so-called apolarity, see \cite{D}). A canonical basis of $R_{d}^*$ is given by the $r_d$ derivations:
$$ \frac{\partial^d}{\partial x_0^{i_{0}}\ldots \partial x_n^{i_n}} \,\, \mathrm{with}\,\, i_0+\ldots +i_n=d.$$

Let $I=(F_1, \ldots, F_r)\subset R$ an ideal generated by $r$ forms of degree $d$. Note that  $F_1,\ldots,F_r$ are points in $\p(R_{d}^*)$.
 We denote by $I_d$ the vector subspace of $R_d$ generated by the $F_1,\ldots,F_r$ and by $I_{d+i}=R_iF_1+\cdots+R_iF_r$, for any $i\geq 0$.
Let us introduce the orthogonal vector space to $I_{d+i}$
$$ I_{d+i}^{\perp}=\{\delta \in R_{d+i}^*\,|\,\delta(F)=0, \,\, \forall F\in I_{d+i}\}.$$
It gives an exact sequence of vector spaces
$$  \begin{CD}
0 @>>> I_{d+i}^{\perp} @>>>R_{d+i}^* @>>> I_{d+i}^* @>>> 0
\end{CD}$$
and  the corresponding  projection map
$$\begin{CD}
\pi_{I_{d+i}}: \p(R^*_{d+i})/\p(I^*_{d+i}) \rightarrow \p(I^\perp _{d+i})
  \end{CD}
$$
Of course one can identify $R_{d+i}/I_{d+i}\simeq (I_{d+i}^{\perp})^*$ and write the decomposition
$R_{d+i}=I_{d+i} \oplus (I_{d+i}^{\perp})^*.$

\begin{rem*}
In the following two  situations, the vector
space   $(I_d^{\perp})^*$ is easy to describe:
\begin{enumerate}
 \item When $I_d$ is generated by $r$ monomials of degree $d$,   $(I_{d}^{\perp})^*$ is  generated by
the  remaining $r_d-r$ monomials.
\item When $I_d=(L_1^d,\ldots,L_r^d)$ where  $[L_i]\in \p(R_{1}^*)$,
$(I_{d}^{\perp})^*$ is generated by degree $d$ polynomials that vanish at the points
$[L_i^{\vee}]\in \p(R_{1}).$
\end{enumerate}
\end{rem*}
It is well known that the tangent spaces to the Veronese varieties can be interpreted as singular hypersurfaces.
More precisely a hyperplane  containing the tangent space $T_{[L^t]}v_t(\p^n)$ corresponds in the dual space $\p^{n\vee}$ to  a hypersurface
of degree $t$ that is singular at the point $[L^{\vee}]$. More generally
a hyperplane  containing the  $s$-th ($s\le 1$) osculating space $T_{[L^t]}^{s}v_t(\p^n)$ corresponds to a hypersurface of degree $t$ and multiplicity $(s+1)$ at the point $[L^{\vee}]$ (see for instance \cite{BCGI}).

\smallskip

Thus the dual variety of $v_t(\p^n)$ is the discriminant variety that parametrizes the  singular  hypersurfaces of degree $t$ when
 the $s$-th osculating variety of $v_t(\p^n)$  parametrizes the hypersurfaces of degree $t$ with a point of multiplicity $s+1$.

We propose now an extended version of the ``main'' theorem of \cite{MMO} (to be precise Theorem 3.2).
\begin{thm}
\label{th1bis}
Let $I=(F_1, \ldots, F_r) \subset R$ be an  artinian ideal generated by $r$ homogeneous polynomials  of degree $d$.
Let $i,k,\delta$ be integers such that $
i\ge 0$, $k\ge 1$.
Assume that there is no syzygy  of degree $i$  among the $F_j$'s.
The following conditions are equivalent:
\begin{enumerate}
\item The ideal $I$ fails the SLP at the range  $k$ in degree $d+ i-k$.
\item There exist  $N^{+}+\delta$, with $\delta \ge 1$,  independent vectors  $(G_{1j},\ldots, G_{rj})_{j=1, \ldots,N^{+}+\delta} \in R_i^{\oplus r}$ and
$N^{+}+\delta $ forms $G_j\in R_{d+ i-k}$ such that
$G_{1j}F_1+ \ldots + G_{rj}F_r=L^kG_j$ for a general linear form $L$ of $\p^n$.
\item
 The $n$-dimensional variety $\pi_{I_{d+i}}(v_{d+i}(\p^n))$ satisfies
$\delta \ge 1$ Laplace equations of order $d+i-k$.
\item \label{item_iv_thm} For any  $L\in R_1$, $\mathrm{dim}_{\C}((I_{d+i}^{\perp})^*\cap \mathrm{H}^0(\cI_{L^{\vee}}^{d+i-k+1}(d+i))\ge N^{-}+\delta$, with $\delta \ge  1$.
\end{enumerate}
\end{thm}

\begin{proof}
The equivalence  $(1)\Leftrightarrow (2)$ is proved  in Theorem \ref{p1}.

Since  $I$ is generated in degree $d$, the map $R_i\times I_d \rightarrow I_{d+i}$ is surjective and the relation $G_1F_1+ \ldots + G_rF_r=L^kG$
is equivalent to  $\p(I_{d+i}^*) \cap T_{[L^{d+i}]}^{d+i-k}v_{d+i}(\p^n)\neq \emptyset$. More generally the
number of independent relations $G_{1j}F_1+ \ldots + G_{rj}F_r=L^kG_j$ is the dimension of the kernel
of the map $\mathrm{H}^0(\Phi_{I,L^k})$ i.e. the dimension of $\mathrm{H}^0(K\otimes \sO_{L^k}(i))$; this number of independent relations, written
in a geometric way, is
$$N^++\delta=\mathrm{dim}[\p(I_{d+i}^*) \cap T_{[L^{d+i}]}^{d+i-k}v_{d+i}(\p^n)]+1,\,\,\, (\delta \ge 0)$$
where the projective dimension is $-1$ if the intersection is empty.  The number \textcolor[rgb]{0.98,0.00,0.00}{$\delta$} is the number of (non trivial) Laplace equations. Indeed,
 the dimension of the $(d+i-k)$-th osculating space to $\pi_{I_{d+i}}(v_{d+i}(\p^n))$ is $r_{d+i-k} -N^{+} -\delta$   since
 the $(d+i-k)$-th osculating space to $v_{d+i}(\p^n)$  meets the center of projection along a $\p^{N^{+}+\delta-1}$.  In other words, the $n$-dimensional variety  $\pi_{I_{d+i}}(v_{d+i}(\p^n))$
 satisfies $\delta$ Laplace equations and  $(3)$ is equivalent to $(2)$.

The image  by $\pi_{I_{d+i}}$ of the  $(d+i-k)$-th osculating space to the Veronese
$v_{d+i}(\p^n)$ in a general point  has codimension $\mathrm{h}^0(K\otimes \sO_{L^k}(i))-N^+$ in
$\p(I_{d+i}^{\perp})$. The codimension corresponds to the number of hyperplanes in $\p(I_{d+i}^{\perp})$ containing the osculating space to
$\pi_{I_{d+i}}(v_{d+i}(\p^n))$. These hyperplanes are images by $\pi_{I_{d+i}}$ of hyperplanes in $\p(R_{d+i}^*)$ containing
$\p(I_{d+i}^*)$ and the $(d+i-k)$-th osculating plane to
$v_{d+i}(\p^n)$ at the point $[L^{d+i}]$.
In the dual setting it means that these hyperplanes are
forms  of degree $d+i$ in $(I_{d+i}^{\perp})^*$ with multiplicity  $(d+i-k+1)$ at $[L^{\vee}]$. It proves that
$(3)$ is  equivalent to $(4)$.

To summarize, the number of Laplace equations is $\hh^0(K\otimes \sO_{L^k})-N^+$ and
$\mathrm{coker}(\mathrm{H}^0(\Phi_{I,L^k}))\simeq (I_{d+i}^{\perp})^*\cap \mathrm{H}^0(\cI_{L^{\vee}}^{d+i-k+1}(d+i)).$
\end{proof}
\begin{rems*}
\begin{enumerate}[1.]
\item \label{cor_cones} Let us explain the geometric meaning of Theorem \ref{th1bis} \ref{item_iv_thm} in a simple case: if $N^-=0$, then \ref{item_iv_thm} means that $I$ fails the SLP at the range $k$ in degree $d+i-k$ if and
only if there exists at any point $M\in\p^n$ a hypersurface of degree $d+i$ with multiplicity $d+i-k+
1$ at $M$ given by a form in
$(I_{d+i}^{\perp})^*\simeq R_{d+i}/I_{d+i}$.
\item Let $
I=(L_1^{d},\ldots,L_r^{d})$ where $L_1, \ldots, L_r$ are general linear forms.
The vector space $(I_{d+i}^{\perp})^*$, where
$I_{d+i}=L_1^{d}R_i+\ldots+L_r^{d}R_i,$ is the vector space of the forms of degree $d+i$
vanishing in $r$ points $[L_j^{\vee}]$ with multiplicity $(i+1)$.
In other words $f\in \cap_{j=1}^r \mathrm{H}^0(\cI_{L_j^{\vee}}^{i+1}(d+i)$) (see
\cite[Corollary 3]{EI}). Geometrically it can be described
as $\p(I_{d+i}^*)=\mathrm{Join}(T_{[L_1^{d+i}]}^{i}v_{d+i}(\p^n), \cdots ,
T_{[L_r^{d+i}]}^{i}v_{d+i}(\p^n)).$
\item By the theorem above, when $
N(r,i,k,d)\ge 0$, the ideal $
I=(L_1^{d},\ldots,L_r^{d})$ fails
the SLP at the range $k$ in degree $d-k+i$ if and only if
the following intersection is not empty:
$$ \mathrm{Join}(T^i_{[L_1^{d+i}]} v_{d+i}(\p^n), \cdots , T^i_{[L_r^{d+i}]}v_{d+i}(\p^n))
\cap \, T^{d+i-k}_{[L^{d+i}]}v_{d+i}(\p^n).$$
\item
Here we focus the attention also on the number $\delta$ of Laplace equations satisfied by
$\pi_{I_{d+i}}(v_{d+i}(\mathbb P^n))$. The geometric meaning of this number was highlighted by
Terracini \cite{Te}
for Laplace equations of order $2$ and recently for any order by \cite{DDI}, where a classification
of varieties satisfying \lq\lq many\rq\rq\ Laplace equations is given.
\end{enumerate}
\end{rems*}
The characterization of the failure of the SLP by the existence of ad-hoc singular hypersurfaces allows us to answer, in the three following propositions,  some questions posed by Migliore and Nagel.
Let us recall their questions:
 \begin{problem*}
 \cite[Problem 5.4]{MN}
Let $I = (x_1^N,x_2^N,x_3^N,x_4^N,L^N)$ for a general
linear form $L$. $R/I$ fails the WLP, for $N = 3, . . . , 12$.
There are some natural questions arising from this example:
\begin{enumerate}
  \item \label{problem1} Prove the failure of the WLP in previous example for all $N \geq 3$.
  \item What happens for mixed powers?
  \item \label{problem3} What happens for almost complete intersections, that is, for $r+1$ powers of general
linear forms in $r$ variables when $r \ge 4$?
\end{enumerate}
  \end{problem*}
  \begin{conj*}\cite[Conjecture 5.13]{MN}
   Let $L_1, \ldots , L_{2n+2}$ be general linear
forms and $I = (L_1^d, \ldots , L_{2n+2}^d)$
\begin{enumerate}
  \item  \label{conj1} If $n = 3$ and$ d = 3$ then $R/I$ fails the WLP.
\item If $n \geq 4$ then $R/I$ fails the WLP if and only if $d > 1$.
\end{enumerate}
  \end{conj*}
We prove \ref{problem1} of  \cite[Problem 5.4]{MN} in Proposition \ref{pr54-1}, \ref{problem3} of \cite[Problem 5.4]{MN}, for $r=4$ and $N=4$,  in Proposition \ref{pr54-2} and \ref{conj1} of \cite[Conjecture 5.13]{MN} in Proposition \ref{pr54-3}.

Since all these results concern powers of linear forms, let us first verify that the hypothesis on the global syzygy in Theorem \ref{th1bis} is not restrictive.
%
\begin{lem}
 \label{lem-syz}
Let $I$ be the ideal $(L_1^{d},\ldots,L_r^{d})$  where the $L_j$ are linear forms and $r<r_d$. Let $K$ be its syzygy bundle. Then
$$ \hh^0(K(i))=0 \Leftrightarrow rr_i\le r_{d+i}.$$
\end{lem}
\begin{proof} One direction is obvious. Let us assume that $rr_i\le r_{d+i}$ and that there exists a relation
$$G_{1}L_1^{d}+ \ldots + G_{r}L_r^{d}=0,  $$
with $G_{1}, \ldots, G_{r}$ forms of $R_i$.  Both hypotheses imply that the projective space
$\mathrm{Join}(T_{[L_1^{d+i}]}^{i}v_{d+i}(\p^n), \cdots , T_{[L_r^{d+i}]}^{i}v_{d+i}(\p^n))$ has dimension strictly less than the expected one.
Since the linear forms are general, it implies that the algebraic closure of $\cup_{L\in R_1^{d+i}}T_{[L^{d+i}]}^{i}v_{d+i}(\p^n)$ has not the expected dimension.
It contradicts the lemma 3.3 in  \cite{BCGI}.
\end{proof}
Proposition \ref{pr54-2} is already proved in \cite[Lemma 4.8]{HSS} and also in
\cite[Theorem 4.2 (ii)]{MMN}. We propose here a new proof based  on the existence of a
singular hypersurface characterizing the failure of the SLP. Let us mention that,
on $\p^2$ a hypersurface of degree $d+i$ with a point of multiplicity $d+i$ is simply an union of
lines (as, for instance, in Theorem \ref{th3} and Proposition \ref{th4}), but on $\p^n$, with $n>2$, a hypersurface
of degree $d+i$ with a point of multiplicity $d+i$ is more generally a cone over a hypersurface in
the hyperplane at infinity. This is the key argument in the proofs of the three following propositions.
\begin{prop}
\label{pr54-1}
 Let $N$ be an integer such that $N\ge 3$.
Then the ideal $(x_0^N,x_1^N,x_2^N,x_3^N, (x_0+x_1+x_2+x_3)^N)$  fails the WLP in degree $2N-3$.
\end{prop}
\begin{rem*}
Of course it is equivalent to say that $(L_1^N,\ldots, L_5^N)$  fails the WLP in degree $2N-3$ for $L_1, \ldots, L_5$ general linear forms.
\end{rem*}
\begin{proof}
Let us consider the syzygy bundle associated to the linear system
$$  \begin{CD}
 0@>>> K @>>>  \sO_{\p^{3}}^{5} @>(x_0^N,x_1^N,x_2^N,x_3^N, (x_0+x_1+x_2+x_3)^N)>>
 \sO_{\p^{3}}(N) @>>> 0.
\end{CD}$$
Since $5r_{N-2}< r_{2N-2}$  Lemma \ref{lem-syz} implies $\hh^0(K(N-2))=0.$
Let $L$ be a linear form. When $N\ge 3$ we have  $5\mathrm{h}^0(\sO_{L}(N-2))\ge \mathrm{h}^0(\sO_{L}(2N-2))$.
According to  Theorem \ref{th1bis} the failure of the WLP in degree $2N-3$ is equivalent to the existence of a surface
with multiplicity $N-1$ in the points $P_0,P_1,P_2,P_3$ and $P(1,1,1,1)$ and multiplicity $2N-2$ at  a moving point $M$.
The five concurrent lines in $M$ passing through $P_0,P_1,P_2,P_3,P$ belong to a quadric cone with equation $\{F=0\}$ (the cone over the conic
at infinity through the five points). Since $F^{N-1}\in \mathrm{H}^0(\cI_{M}^{2N-2}(2N-2))$ the hypersurface
$\{F^{N-1}=0\}$ has the  desired properties.
\end{proof}
In $\p^n$ there is always a quadric through $\frac{n(n+3)}{2}$ points in general position. Then given any general point
$M\in \p^{n+1}$, there is a quadratic cone with a vertex at $M$ and passing through $\frac{n(n+3)}{2}$ fixed points  in general position.
Then we prove,
\begin{prop}
\label{pr54-2}
In the following cases  the ideal $(L_1^N,\ldots, L_{\frac{n(n+3)}{2}}^N)$ fails the WLP in degree $2N-3$:
\begin{itemize}
 \item    $N= 3$ and $n\ge 2$,
\item $N=4$ and $2\le n\le 4$,
\item $N>4$ and $2\le n\le 3$.
\end{itemize}
\end{prop}
\begin{proof}
 Let us consider the syzygy bundle associated to the linear system
$$  \begin{CD}
 0@>>> K @>>>  \sO_{\p^{n+1}}^{\frac{n(n+3)}{2}} @>>>
 \sO_{\p^{n+1}}(N) @>>> 0.
\end{CD}$$
Let $L$ a linear form. Then the inequality   ${\frac{n(n+3)}{2}}\mathrm{h}^0(\sO_{L}(N-2))\ge \mathrm{h}^0(\sO_{L}(2N-2))$
is true if and only if $N$ and $n$ are one of the possibilities stated in the theorem. In all these cases we have
$ \frac{n(n+3)}{2}r_{N-2}\le r_{2N-2}$, and by  Lemma \ref{lem-syz}, $\hh^0(K(N-2))=0$.

According to  Theorem \ref{th1bis} the failure of the WLP is equivalent to the existence of a hypersurface
with multiplicity $N-1$ in the points $[L_i^{\vee}]$ and multiplicity $2N-2$ at the moving point $M$.
The lines through $M$ and $[L_i^{\vee}]$  belong to a quadratic cone with equation $\{F=0\}$ (the cone over the quadric
at infinity through the  points). Since $F^{N-1}\in \mathrm{H}^0(\cI_{M}^{2N-2}(2N-2))$ the hypersurface
$\{F^{N-1}=0\}$ has the desired properties.
\end{proof}
\begin{prop}
\label{pr54-3}
The ideal
$I=(L_1^3, \ldots, L_8^3)$ fails the WLP in degree $3$ where
 $L_1, \ldots, L_8$ be general linear forms on $\p^6$.
\end{prop}
\begin{proof}
Since $8r_{1}< r_{4}$   Lemma \ref{lem-syz} implies $\hh^0(K(1))=0.$
We have to prove that, on  a general hyperplane $L$, the cokernel of
$  \begin{CD}
  \HH^0(\sO_{L}(1))^{8} @>>>
 \HH^0(\sO_{L}(4))
\end{CD}$
has dimension strictly greater than $  \hh^0(\sO_{L}(4)) -\hh^0(\sO_{L}(1))^{8}=78.$
The dimension of this cokernel is the dimension of  the quartics with a quadruple point
$[L^{\vee}]$  and $8$ double points. We consider on the hyperplane  at infinity the vector space $V$ of quadrics through the images of the $8$ points
$[L_1^{\vee}], \ldots, [L_8^{\vee}]$. It has dimension $13$. Let $Q_1, \ldots, Q_{13}$ be a basis of this space of quadrics. Then the vector space
$\mathrm{Sym}^2(V)$ of quartics generated by the products $Q_iQ_j$ has dimension $91$ and all these quartics are singular in the $8$ points.
In $\p^6$ the independent quartic cones over these quartics belong to the cokernel.
\end{proof}
In the next section, we propose many examples of ideals failing the WLP or the SLP by producing ad-hoc singular hypersurfaces.
\section{Classes of ideals failing the WLP and the SLP}
 \label{s4}
\subsection{Monomial ideals coming from singular hypersurfaces}
In their nice paper about osculating spaces of Veronese surfaces, Lanteri and Mallavibarena remark that the equation of the curve given by three
concurrent lines depends only on six monomials instead of seven. More precisely let us consider  a cubic with a triple point at $(a,b,c)$ passing through
$P_0$, $P_1$ and $P_2$. Its equation is $(bz-cy)(az-cx)(ay-bx)=0$ and it depends only on the monomials
$x^2y,xy^2,x^2z,xz^2,y^2z, yz^2$. So there is a non zero form in
$$(I_{3}^{\perp})^*=<x^2y,xy^2,x^2z,xz^2,y^2z, yz^2>\simeq\frac{R_3}{<x^3,y^3,z^3, xyz>}$$ that is triple at a general point. In this way they explain
the Togliatti surprising phenomena (\cite[Theorem 4.1]{LM}, \cite{I2} and \cite{FI}).

We apply this idea in our context. Recall that in the monomial case being artinian to the ideal $I$ means
that it contains the forms  $x_0^d, \ldots, x_n^d$.  Let us consider the $(n+1)$ fundamental points
$P_0, P_1, \ldots, P_n$ and let us assume  that the number $r$ of monomials generating  $I$
is chosen such  that $N(r,i,k,d)=0$ for  $i\ge 0,\,\, k\ge 1$ fixed integers.
 Then, as it is noted in item \ref{cor_cones} of Remarks after Theorem \ref{th1bis},
the ideal $I$ fails the SLP at the range $k$ in degree $d+i-k$ if and only if there exists  at any point
$M$  a hypersurface of degree $d+i$  with multiplicity $d+i-k+1$ at $M$ given by a form in $(I_{d+i}^{\perp})^*\simeq R_{d+i}/I_{d+i}$.
 We have to write this equation with a number of monomials as small as possible. Then the orthogonal space becomes bigger and we will cover all the possible choices.

First of all we  describe exhaustively  the   monomial ideals $(x^4,y^4, z^4, f,g)\subset \C[x,y,z]$ of degree $4$  that do not verify the WLP.
\begin{thm}
 \label{th3}
Up to permutation of variables the monomial ideals generated by five quartic forms in $\C[x,y,z]$ that fail the WLP in degree $3$ are
 the following
\begin{itemize}
 \item   $I_1= (x^4,y^4,z^4, x^3z, x^3y)$,
\item  $ I_2=(x^4,y^4,z^4, x^2y^2, xyz^2)$.
\end{itemize}
\end{thm}
\begin{rem*}
 Geometrically it is evident that the first ideal  $ (x^4,y^4,z^4, x^3z, x^3y)$ fails the WLP. Indeed
  under the Veronese map, a linear form $L$ becomes a rational normal curve of degree four that defines a projective space $\p^4$ and, modulo $L$, the restricted monomials $\bar{x}^i\bar{y}^j$ can be interpreted as points of this $\p^4$.
 Then
the tangent line to the rational quartic curve  at the point $[\bar{x}^4]$ contains
the two points $[\bar{x}^3\bar{y}]$ and $[\bar{x}^3\bar{z}]$. This line  meets the plane $\p(<\bar{x}^4,\bar{y}^4,\bar{z}^4>)$ in one point; it implies that
$$\mathrm{dim}_{\C}<\bar{x}^4,\bar{y}^4,\bar{z}^4,\bar{x}^3\bar{y},\bar{x}^3\bar{z} >\le 4.$$
For the second ideal, it is not  evident  to see  that the line $\p(<\bar{x}^2\bar{y}^2,\bar{x}\bar{y}\bar{z}^2>)$ always (for any restriction) meets the plane $\p(<\bar{x}^4,\bar{y}^4,\bar{z}^4>)$.
\end{rem*}
\begin{proof}
Let us consider the points $P_0$ , $P_1$ and $P_2$ and
 the degree $4$ curves with a quadruple point in $(a,b,c)$ passing through these three points. These curves are product
of four lines:
$$ f(x,y,z)= (ay-bx)(cx-az)(cy-bz)(\alpha(ay-bx)+\beta(cx-az))=0.$$
 \begin{figure}[h!]
    \centering
    \includegraphics[height=5cm]{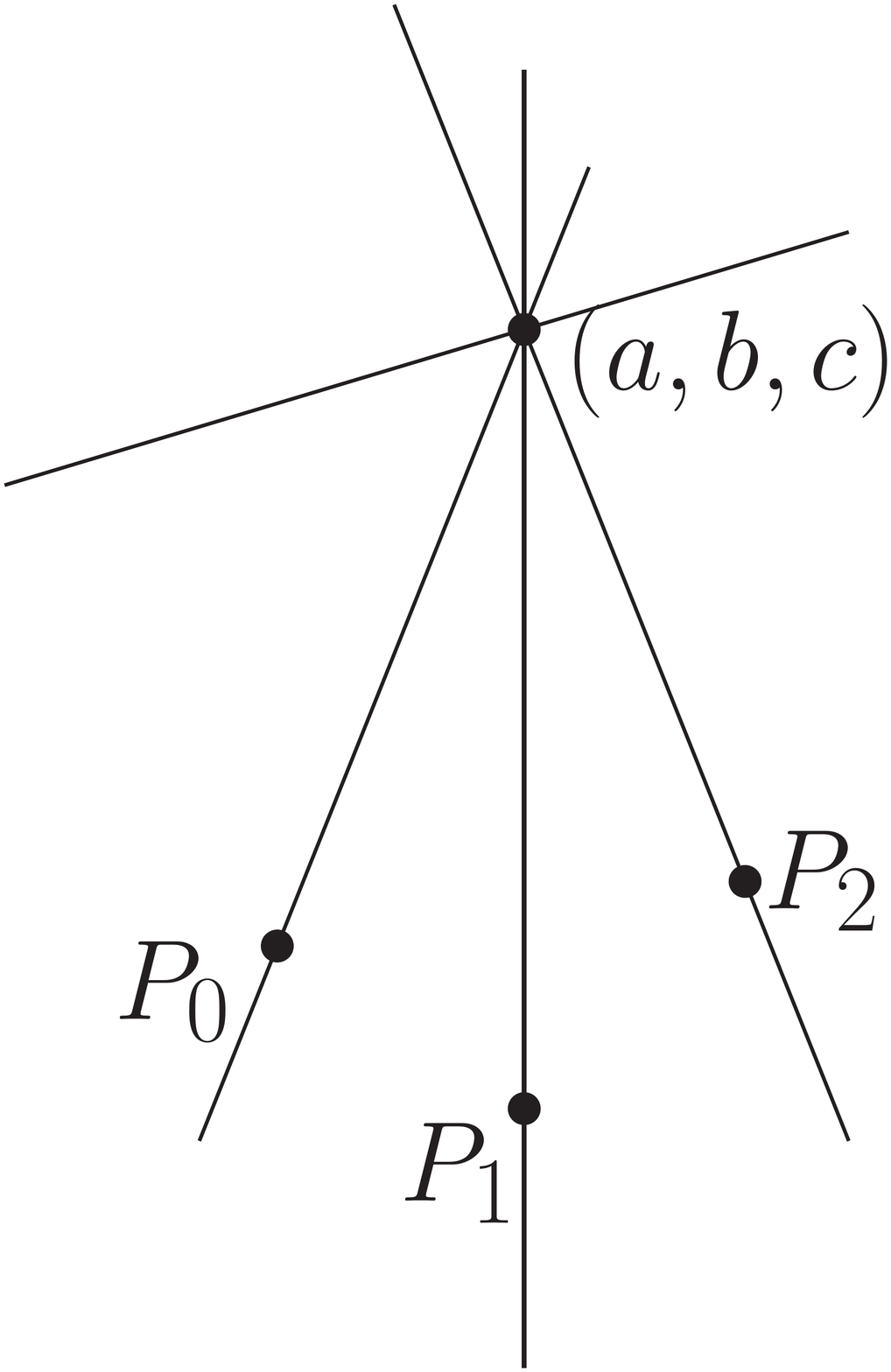}
    \caption{Quartic with a quadruple point}
  \end{figure}
Expanding $f$ explicitly in the coordinates $(x,y,z)$, we see that the forms $x^4,y^4,z^4$ are missing
and that
twelve monomials appear to write its equation.
Since we want only ten monomials, we have to remove two. The following possibilities occur:
\begin{itemize}
 \item $\alpha=0$ (or equivalently by permutation of variables  $[\beta=0]$ or $[\alpha\neq 0$, $\beta \neq 0$ and $b\alpha =c\beta]$) then  the remaining linear system is
$ (x^4,y^4,z^4, y^3z, xy^3).$ It corresponds to the first case i.e. to the ideal $I_1$.
\item  $\alpha\neq 0 $ and $\beta \neq 0$ but $c\beta +b\alpha=0$ (or equivalentely by permutation of variables $[2b\alpha-c\beta=0]$ or
$[b\alpha-2c\beta=0]$) then the remaining linear system is
$ (x^4,y^4,z^4, x^2yz, y^2z^2).$ It corresponds to the second  case i.e. to the ideal $I_2$.
\end{itemize}
\end{proof}

\begin{rem*} The quartic curve with multiplicity four in $(a,b,c)$ consists,  in the first case,  of two lines and a double line that are concurrent; in the second case  of four concurrent lines in harmonic division.
\end{rem*}

We  do not apply the same technique to describe exhaustively  the  monomial ideals $(x^5,y^5, z^5, f,g,h)\subset \C[x,y,z]$ of degree $5$
 that do not verify the WLP because the computations become too tricky. But we can give some cases by geometric arguments.
\begin{prop}
 \label{th4}
The following monomial ideals
\begin{itemize}
 \item $ (x^5,y^5,z^5, x^3y^2, x^3z^2,x^3yz)$,
\item  $ (x^5,y^5,z^5, x^4z, x^4y,m)$, where $m$ is any monomial,
\item $ (x^5,y^5,z^5, x^3y^2, x^2y^3,x^2y^2z)$,
\end{itemize}
fail the $\mathrm{WLP}$ in degree $4$.
\end{prop}
\begin{proof}
  Under the Veronese map, a linear form $L$ becomes a rational normal curve of degree five that defines a projective space $\p^5$ and, modulo $L$, the restricted monomials $\bar{x}^i\bar{y}^j$ can be interpreted as points of this $\p^5$. Then
the tangent line to the rational quintic curve  at the point $[\bar{x}^5]$ contains
the two points $[\bar{x}^4\bar{y}]$ and $[\bar{x}^4\bar{z}]$. This line  meets the plane $\p(<\bar{x}^5,\bar{y}^5,\bar{z}^5>)$ in one point; it implies that
$$\mathrm{dim}_{\C}<\bar{x}^5,\bar{y}^5,\bar{z}^5,\bar{x}^4\bar{y},\bar{x}^4\bar{z}, \bar{m} >\le 5.$$
In the same way the osculating plane at $[\bar{x}^5]$ i.e. $\p(<\bar{x}^3\bar{y}^2,\bar{x}^3\bar{z}^2,\bar{x}^3\bar{y}\bar{z} >)$ meets the plane  $\p(<\bar{x}^5,\bar{y}^5,\bar{z}^5>)$ in one point.

\noindent In the last case, the geometric argument is not so evident. Let us set $X=bz-cy$ and
$Y=cx-az$. Then the equation of the product of the five concurrent lines is $f(X(x,y,z),Y(x,y,z))=XY(aX+bY)(\alpha X+\beta Y) (\gamma X+\delta Y)=\\
a\alpha \gamma X^4Y+(a\alpha \gamma +b\alpha \gamma +a\alpha \delta)X^3Y^2+
(b\beta \gamma +b\alpha \delta +a\beta \delta)X^2Y^3 +b\beta \delta XY^4=0.$

\noindent For any point $M(a,b,c,d)$ we choose $(\alpha, \beta, \gamma, \delta)$ such that  $a\alpha \gamma +b\alpha \gamma +a\alpha \delta=0$ and
$b\beta \gamma +b\alpha \delta +a\beta \delta=0$.
Then the equation depends only on $15$ monomials  and the remaining monomials are
$(x^5,y^5,z^5, x^3y^2, x^2y^3,x^2y^2z).$
\end{proof}

We  describe now some monomial ideals in $ \C[x,y,z,t]$, generated in degree $3$, that do not verify the WLP.
\begin{prop}
 \label{th3-1}
The monomial ideals
 $I=(x^3,y^3,z^3,t^3,  f_1,f_2,f_3,f_4,f_5,f_6)$ where the forms $f_i$ are chosen among one of the following sets of monomials:
\begin{itemize}
\item  $ \{ x^2y, xy^2,x^2z, x^2t, y^2z, y^2t,z^2t, zt^2, xyz, xyt \},$ (Case (A1))
\item  $ \{ x^2y, xy^2, xz^2, y^2z, yz^2,  y^2t, zt^2, z^2t \},$ (Case (A2))
\item  $ \{ x^2y,xy^2, z^2t,zt^2, xyz,xyt,xzt,yzt\},$ (Case (A3))
\item  $ \{ xz^2, yz^2,xyz,xyt, x^2y,xy^2,z^2t,zt^2 \},$ (Case (A4))
\item  $ \{ x^2y, xy^2, x^2z,xz^2, x^2t, xt^2, xyz, xzt, xyt,yzt \},$ (Case (B1))
\end{itemize}
fail the WLP in degree $2$.
\end{prop}
\begin{rem*}
We do not know if under permutation of variables the description above is exhaustive or not.  The singular cubic that we are considering here are union of concurrent planes and not
all the cubic cones.
\end{rem*}
\begin{proof}
We look for  a surface of degree $3$ with multiplicity $3$   at a general point $M(a,b,c,d)$  that passes through the  points
$P_0, P_1, P_2, P_3$
such that its equation  depends only on the remaining monomials
in $R_3/I_3$. Such a cubic surface is a cone over a cubic curve. Here, instead of a general cubic cone we consider only three concurrent planes.
Since these $3$ planes have to pass through $P_0, P_1, P_2$ and  $P_3$ it remains only, after a simple verification,  the following cases:
  \begin{figure}[h!]\centering
   \begin{subfigure}[b]{0.3\textwidth}
    \includegraphics[height=2.5cm]{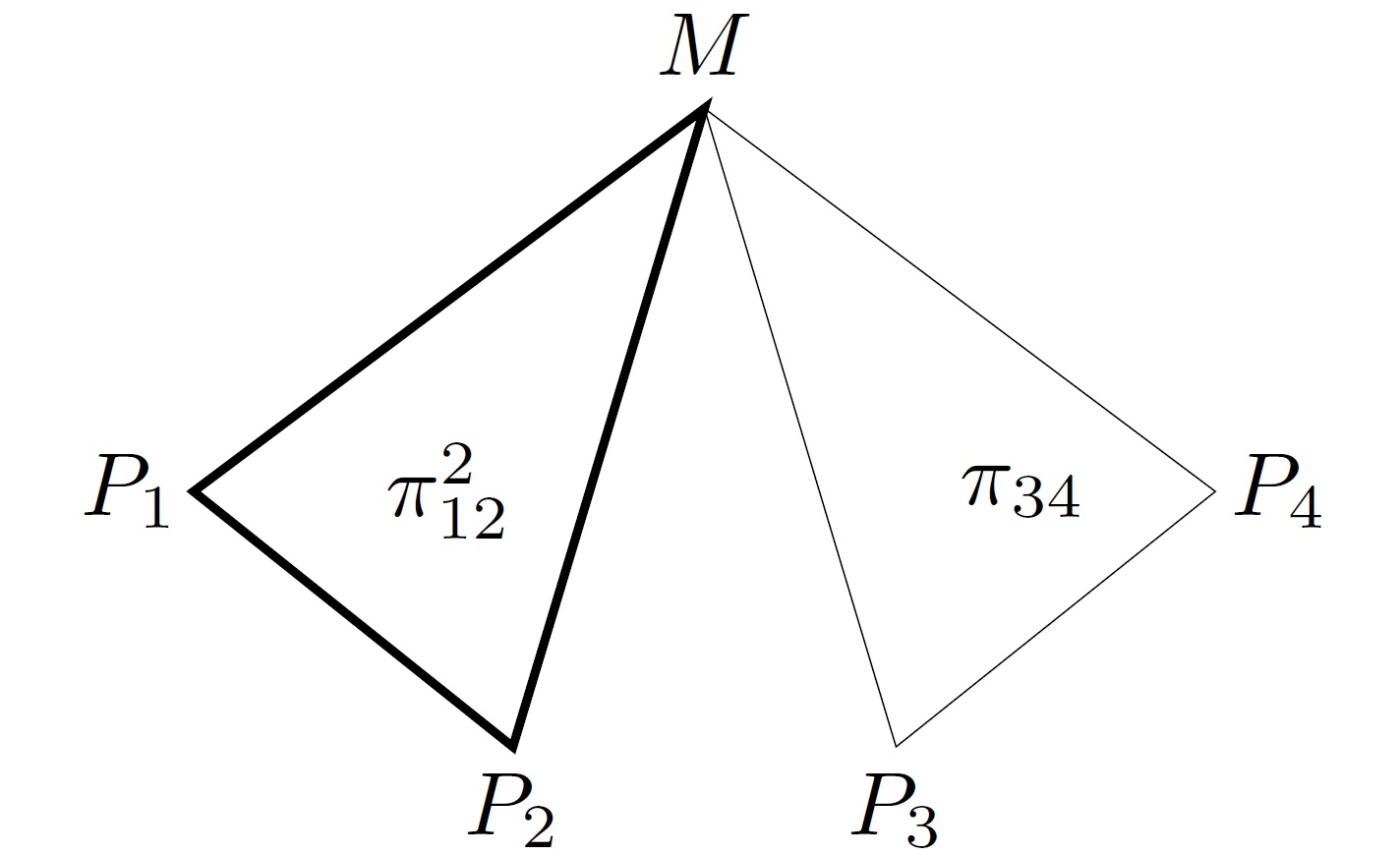}
  \caption{Case (A1).}\label{pi12^2-pi34}
 \end{subfigure}
  \begin{subfigure}[b]{0.3\textwidth}
      \includegraphics[height=2.5cm]{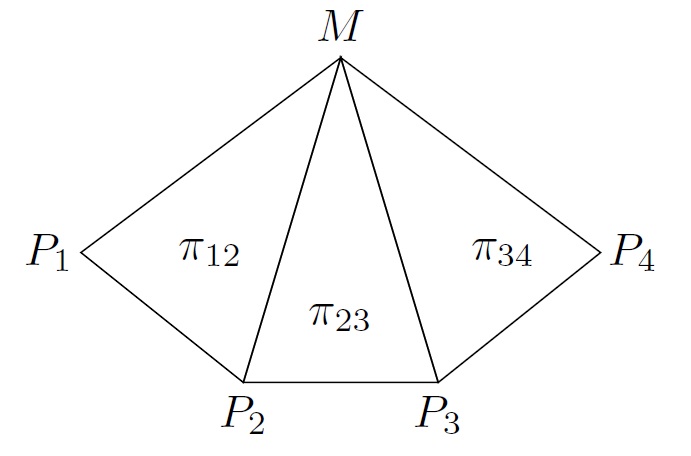}
    \caption{Case (A2).}\label{pi12-pi23-pi34}
    \end{subfigure}
   \\
    \begin{subfigure}[b]{0.3\textwidth}
       \includegraphics[height=2.5cm]{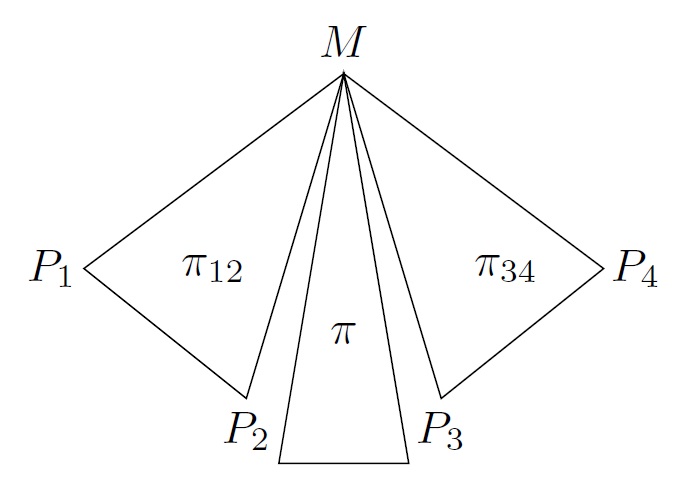}
      \caption{Case (A3).}\label{pi12-pi34-pi}
      \end{subfigure}
       \begin{subfigure}[b]{0.3\textwidth}
        \includegraphics[height=2.5cm]{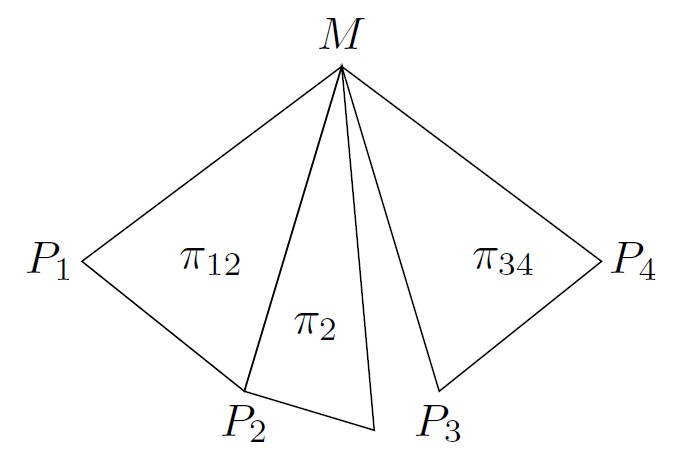}
    \caption{ Case (A4).}\label{pi12-pi2-pi34}
    \end{subfigure}
  \\
\begin{subfigure}[b]{0.3\textwidth}
    \includegraphics[height=3cm]{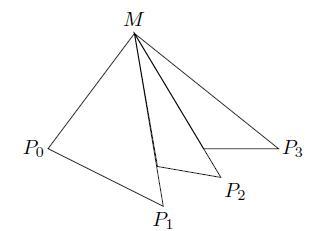}
    \caption{Case (B1)}
  \end{subfigure}
\end{figure}
\begin{itemize}
 \item (A1)
The equation of the cubic is
$(bx-ay)(dz-ct)^2=0.$
 \item (A2) The equation of the cubic is
$(bx-ay)(dz-ct)(c x-a z)=0.$
\item (A3) The equation of the cubic is
$(bx-ay)(dz-ct)(bx+ay+udz+uct)=0$ where  at any point
$(a,b,c,d)$ the function $u(a,b,c,d) $ verifies $ab+u(a,b,c,d)cd=0$.
\item (A4)  The equation of the cubic is
$(bx-ay)(dz-ct)(bdx+ady-2abt)=0$.
\item  (B1) The equation of the cubic is
$(cy-bz)(dz-ct)(dy-bt)=0.$
\end{itemize}
\end{proof}

If we want $I_3$ to be of dimension $r< 10$ (for instance $r=8$) we need $10-r+1$ independent  cubics
 with a triple point.
 So, to  get the failure of the WLP, we need
$10-r+1$ independent cubics with a triple point.
Let us recover with our method  two linear systems  of eight cubic forms (the complete classification is already done in  \cite[Theorem 4.10]{MMO})
 that fail the WLP in degree $2$.
\begin{prop}
 The following monomial ideals
\begin{itemize}
 \item $I= (x^3,y^3,z^3,t^3,  x^2y, xy^2, zt^2, z^2t),$
\item  $J= (x^3,y^3,z^3,t^3, xyz, xyt, xzt, yzt)$
\end{itemize}
fail the $\mathrm{WLP}$ in degree $2$.
\end{prop}
\begin{rem*}
The ideals $I$ and $J$  correspond respectively to the cases $(3)$ and $(1)$ in  \cite[Theorem 4.10]{MMO}.
\end{rem*}
\begin{proof}
Let us consider  the following three forms defining singular cubics passing through the fundamental points and a general point $(a,b,c,d)$:
$$ (ct-dz)(at-dx)(ay-bx)=0, (ct-dz)^2(ay-bx)=0, (ct-dz)(ay-bx)^2=0.$$
They are particular cases of type $(A)$ in the proof of Proposition \ref{th3-1}. They are linearly independent and can be written with twelve monomials.
Then it remains only $8$ forms for $I_3$:
$$ I=(x^3,y^3,z^3,t^3,  x^2y, xy^2, zt^2, z^2t).$$

Let us consider  the following three forms defining singular cubics passing through the basis points and the general point $(a,b,c,d)$:
$$ (bz-cy)(az-cx)(ay-bx)=0, (bx-ay)(at-dx)(dy-bt)=0, (az-cx)(dx-at)(dz-ct)=0.$$
They are   cases of type $(B1)$ in the proof of  Proposition \ref{th3-1}. They are linearly independent and can be written with twelve monomials:
$$ (x^2y, x^2z, xy^2, xz^2, y^2z, yz^2, t^2y, t^2z, ty^2, tz^2,t^2x, x^2t).$$ It remains only
$$J=(x^3,y^3,z^3,t^3, xyz, xyt, xzt, yzt).$$
\end{proof}
Of course the same argument (concurrent planes or hyperplanes) can be used in degree or dimension bigger than $3$. For instance
let us give a set of monomial ideals in $ \C[x,y,z,t]$,  generated in degree $4$ that fail the WLP.

\begin{prop}
\label{d4m}
Let $f_1, \ldots, f_{11}$ be eleven monomials chosen among
$$x^3y , x^3z, x^3t, xy^3, xz^3, xt^3, y^3z, y^3t, yz^3, yt^3, z^3t, zt^3, x^2y^2, z^2t^2, y^2z^2, x^2t^2.$$
Then the ideal $I=(x^4,y^4,z^4, t^4,f_1, \ldots, f_{11})$ fails the WLP in degree $3$.
\end{prop}
\begin{proof}
At any point $ M=(a,b,c,d)$ an equation of a surface of degree $4$ with multiplicity $4$ at $M$  that passes through the  points
$P_0, P_1, P_2,P_3$ is given by
$$ f(x,y,z,t)=(ct-dz)(at-dx)(ay-bx)(bz-cy)=0.$$
\end{proof}
We conclude this section with  an example that fails the SLP at the range  $2$.
\begin{prop}
\label{d4mslp}
 The ideal $I=(x^4,y^4,z^4, xy^3,xz^3,x^2yz,y^2z^2, y^3z, yz^3)\subset \C[x,y,z] $ fails the SLP at the range $2$ in degree $2$.
\end{prop}
\begin{proof}
Let   $P_0$, $P_1$, $P_2$  and $M(a,b,c)$ be four points. We consider the quartic curve consisting  of the union of
the four lines $(MP_0), (MP_1), (MP_2)$ and $(P_0P_1)$.
It is a quartic passing through $P_0,P_1, P_2$ and triple at $M(a,b,c)$. It depends on the six monomials
$ x^3y,x^3z,x^2y^2, xy^2z, x^2z^2, xyz^2.$
Then it remains $9=15-6$ monomials
 $$I_4=<x^4,y^4,z^4, xy^3,xz^3,x^2yz,y^2z^2, y^3z, yz^3>.$$ The associated syzygy bundle $K$
verifies $h^0(K\otimes \sO_{L^2})\neq 0$ for a general linear form $L$. It proves that
$I=(x^4,y^4,z^4, xy^3,xz^3,x^2yz,y^2z^2, y^3z, yz^3)$ fails the SLP at the range $2$ in degree $2$.
\end{proof}
\subsection{Non monomial examples coming from singular hypersurfaces}
Let us study now the interesting case   $I_d^{\perp}=\mathrm{H}^0(\cI_Z(d))^{*}$ where $Z$ is a finite set of distinct points  in $\p^{2\vee}$ of length $|Z|$ and
$\cI_Z$ its ideal sheaf.
The set $Z$ corresponds by projective duality to a set of $|Z|$ distinct  lines in $\p^2$ defined by linear forms $l_1,\ldots, l_{|Z|}$.
We will now consider the ideal $I\subset R$ generated by $(l_1^d, \ldots, l_{|Z|}^d)$.
We have $|Z|=\mathrm{dim}_{\C}I_d$.

\begin{prop}
\label{nmslp}
Let $k\ge 1$, $r=r_d-r_{d-k}$ and $Z=\{l_1^{\vee},\ldots, l_{r}^{\vee}\}$ a finite set of $r$ distinct points in $\p^{2\vee}$
where $l_i$ are linear forms on $\p^2$. Assume that there exists a subset  $Z_1\subset Z$,
of length $r-d+k-1$,
contained in a curve $\Gamma_{1}$ of degree $k-1$. Then
 the ideal $I=(l_1^d, \ldots, l_{r}^d)$ fails the SLP at the range $k$ in degree $d-k$.
\end{prop}
\begin{proof}
The union of  $\Gamma_{1}$ and  $(d-k+1)$ concurrent lines at a point $P$ passing through the
remaining points $Z\setminus Z_1$, is
a non zero section of
$\mathrm{H}^0(\cI_Z\otimes \cI_P^{d-k+1}(d))$.  By Theorem \ref{th1bis} it proves that $I$ fails the SLP at the range $k$ in degree $d-k$.
\end{proof}


With this method it is always possible to find systems  of any degree that fail the SLP by
exhibiting a curve of degree $d$ with multiplicity $d-k+1$ at a  general point $P$.
But one can find  some set of points  for which these special curves  do not split  as product of lines
(see  Proposition \ref{B3} in the next section).
\section{SLP at the range  $2$ and line arrangements on $\p^2$}
\label{s5}
A line arrangement  is a collection of distinct lines in the projective plane. Arrangement of lines or more generally arrangement of hyperplanes
  is a famous and classical topic that has been studied by many authors for a very long time (see \cite{Cartier} or \cite{OT} for a good introduction).

Let us denote by
$f=0$ the equation of the union of lines of the considered arrangement. Another classical object associated to the arrangement is the
vector bundle $\mathcal{D}_0$ defined as the kernel of the  jacobian map:
$$ \begin{CD}
    0 @>>> \mathcal{D}_0 @>>> \sO_{\p^2}^{3} @>(\partial f)>> \sO_{\p^2}(d-1).
   \end{CD}
$$
The  bundle $\mathcal{D}_0 $ is called \textit{derivation bundle} (sometimes  logarithmic bundle ) of the line arrangement  (see \cite{S}  and \cite{Sc} for an introduction to
derivation bundles).

One can think about the  lines of the arrangement in $\p^{2}$ as a set of distinct points $Z$  in $\p^{2\vee}$. Then we will
denote by  $\mathcal{D}_0(Z)$ the associated derivation bundle.

The arrangement of lines is said {\it free with exponents} $(a,b)$ when its derivation bundle splits on $\p^2$ as a sum of two line bundles, more precisely when
$$ \mathcal{D}_0(Z)=\sO_{\p^2}(-a)\oplus \sO_{\p^2}(-b).$$

The splitting of $\mathcal{D}_0(Z)$ over a line $l\subset \p^2$ is related to the existence of curves (with a given degree $a+1$) passing through $Z$ that are multiple (with multiplicity $a$) at $l^{\vee}\in \p^{2\vee}$.
More precisely,
\begin{lem}(\cite[Proposition 2.1]{V2})
\label{linksd} Let $Z\subset \p^{2\vee}$ be a set of $a+b+1$ distinct points with $1\le a\le b$  and $l$ be
a general line in $\p^{2}$. Then  the following conditions are equivalent:
\begin{enumerate}
 \item $\mathcal{D}_0(Z)\otimes
\sO_{l}=\sO_{l}(-a)\oplus \sO_{l}(-b)$.
 \item  $\mathrm{h}^0((\mathcal{J}_{Z}\otimes \mathcal{J}_{l^{\vee}}^{a})(a+1))\neq
0$ and $\mathrm{h}^0((\mathcal{J}_{Z}\otimes \mathcal{J}_{l^{\vee}}^{a-1})(a))=
0.$
\end{enumerate}
\end{lem}
 In our context it implies  the following characterization of unstability. We recall that
a rank two vector bundle $E$  on $\p^n, n\ge 2$ is  \textit{unstable} if and only if
its splitting $E_l=\sO_l(a)\oplus \sO_l(b)$ on a general line  $l$  verifies $\mid a-b\mid \ge 2$. This characterization
is a consequence of the Grauert-M\"ulich theorem, see \cite{OSS}.

\begin{prop}
\label{th5}
 Let $I\subset R=\C[x,y,z]$ be an  artinian ideal generated by $2d+1$  polynomials $l_1^d, \ldots, l_{2d+1}^d$ where $l_i$ are distinct linear forms in $\p^2$.
Let  $Z=\{l_1^{\vee},\ldots, l_{2d+1}^{\vee}\}$ be the corresponding set of points in $\p^{2\vee}$. Then the following conditions are equivalent:
\begin{enumerate}
\item The ideal $I$ fails the $\mathrm{SLP}$ at the range  $2$ in degree $d-2$.
\item The derivation bundle $\mathcal{D}_0(Z)$   is unstable.
\end{enumerate}
\end{prop}
\begin{proof}
The failure of the SLP at the range $2$ in degree $d-2$ is equivalent to the existence at a general point $l^{\vee}$ of a curve of degree $d$ with multiplicity
$d-1$ at $l^{\vee}$ belonging to $I_d^{\perp}=\mathrm{H}^0(\cI_Z(d))$. By the lemma \ref{linksd} it is equivalent to  the following splitting
 $$\mathcal{D}_0(Z)\otimes \sO_{l}=\sO_{l}(d-s)\oplus \sO_{l}(d+s) \,\, \mathrm{with}\,\, s>0,$$
on a general line  $l$.
In other words  the failure of the SLP is equivalent to have a non balanced  decomposition
and according to Grauert-M\"ulich theorem it is equivalent to unstability.
\end{proof}
Let us give now an ideal generated by non monomials quartic forms  that fails the SLP at the range $2$. It comes from a line arrangement, called B3 arrangement (see  \cite[Pages 13, 25 and 287]{OT}), such that
the associated derivation bundle  is unstable (in fact even decomposed). The existence of a  quartic curve with a general triple point is the key argument. But
contrary to the previous  examples, this quartic is irreducible and  consequently not obtainable by Proposition \ref{nmslp}.
\begin{prop}
\label{B3}
The ideal $$I=(x^4,y^4,z^4,(x+y)^4,(x-y)^4,(x+z)^4,(x-z)^4, (y+z)^4,(y-z)^4)\subset \C[x,y,z] $$
fails  the SLP at the range  $2$ and degree $2$.
\end{prop}
\begin{proof}
Consider the set $Z$ of the nine dual points  of  the linear forms $x,y,z,x+y,x-y,x+z,x-z, y+z,y-z$. Let $I$ be the artinian ideal
$(x^4,y^4,z^4,(x+y)^4,(x-y)^4,(x+z)^4,(x-z)^4, (y+z)^4,(y-z)^4)$ and $K$ its syzygy bundle.
The  derivation bundle of the arrangement is $\mathcal{D}_0(Z)=\sO_{\p^2}(-3)\oplus \sO_{\p^2}(-5)$ (it is free  with exponents $(3,5)$; see \cite{OT} for a proof).
 Then, according to the Lemma \ref{linksd} there is at any point $P$ a degree $4$ curve with multiplicity $3$ at $P$ passing through $Z$.
In other words, by Theorem \ref{th1bis}, $I$ fails  the SLP at the range  $2$ and degree $2$.
 \end{proof}
  \begin{figure}[h!]
    \includegraphics[height=6.5cm]{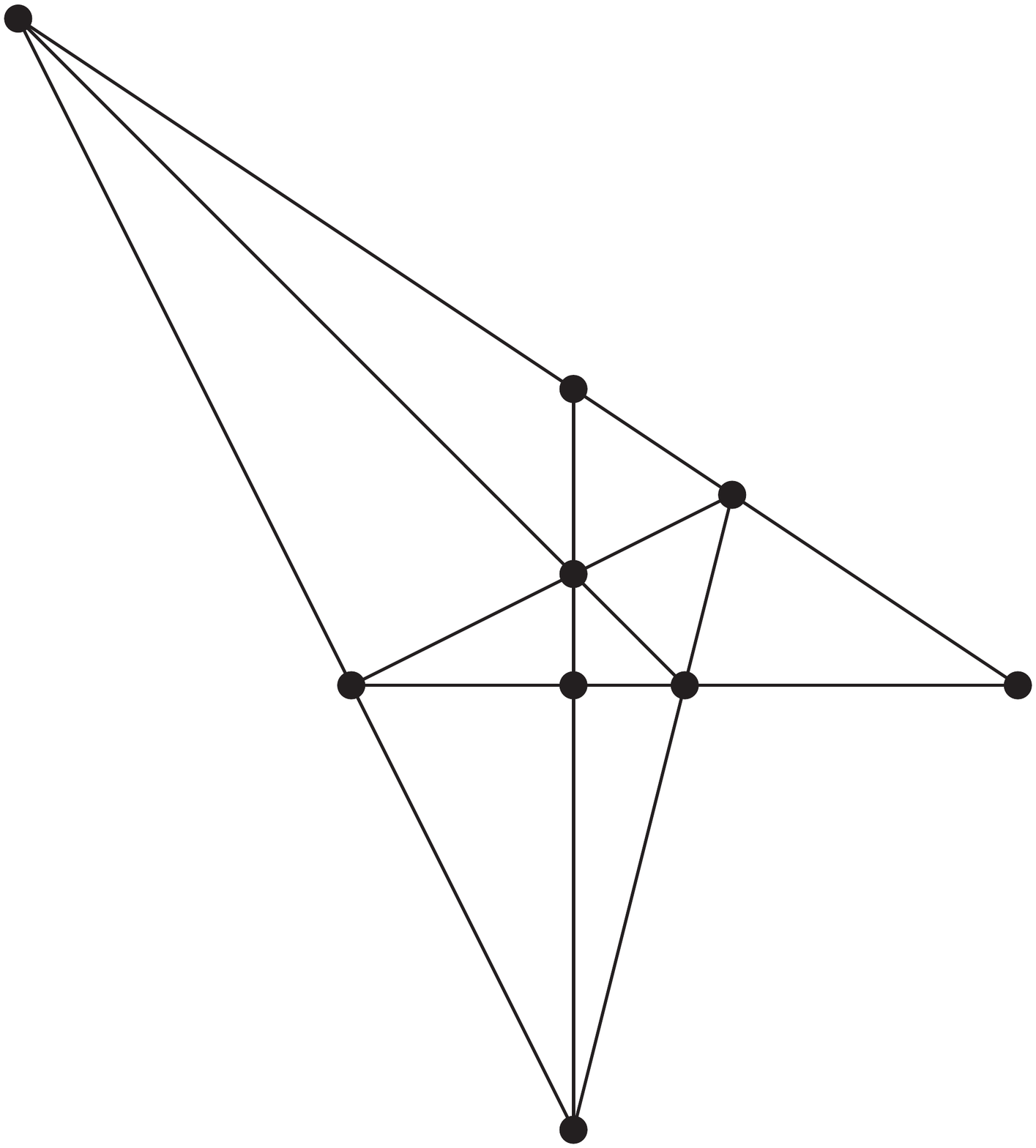}
    \caption{Dual set of points of the $B3$ arrangement}
  \end{figure}
More generally  non balanced  free arrangements lead to ideals that fail the SLP.
\begin{prop}
 Let $\mathcal{A}=\{ l_1, \ldots, l_{a+b+1}\}$ a  line arrangement that is free with exponents $(a,b)$ such that  $a\le b $,  $b-a\ge 2$ and $a+b$ even.
The ideal $I=( l_1^{\frac{a+b}{2}}, \ldots, l_{a+b+1}^{\frac{a+b}{2}})$
fails the SLP at the range $2$ and degree $\frac{a+b}{2}-1$.
\end{prop}
\begin{rem*}
 If $a+b$ is odd we can  add to $Z$ one point $P$ in general  position with respect to $Z$ and we can prove in the same way that
 $I=( l_1^{\frac{a+b+1}{2}}, \ldots, l_{a+b+1}^{\frac{a+b+1}{2}}, (P^{\vee})^{\frac{a+b+1}{2}})$
fails the SLP at the range $2$ and degree  $\frac{a+b}{2}$.
\end{rem*}
\begin{proof}
Let us denote by
$Z=\{ l_1^{\vee}, \ldots, l_{a+b+1}^{\vee}\}$ the dual set of points of $\mathcal{A}$.
Since there exists at any general point $l^{\vee}$ a curve of degree $a+1$ passing through
${Z}$, the Lemma \ref{linksd} implies that $\mathcal{D}_0(Z)$ is unstable and Proposition \ref{th5} implies that $I$
fails the SLP at the range $2$ and degree $\frac{a+b}{2}-1$.
\end{proof}
\subsection{SLP at the range $2$ and Terao's conjecture}
One of the main conjecture about hyperplane arrangements
(still open also for line arrangements) is  Terao's conjecture.
It concerns the  free arrangements.
The conjecture says that freeness depends only on the combinatorics of the arrangement. Let us recall that the combinatorics of the arrangement $\mathcal{A}=\{l_1,\ldots, l_n\}$
is determined by an incidence graph. Its vertices are
  the lines $l_k$ and the  points  $P_{i,j}=l_i\cap l_j$. Its edges are joining $l_k$ to $P_{i,j}$ when $P_{i,j}\in l_k$.
We refer again to \cite{OT} for a good introduction to the subject.  Terao's conjecture is valid not only for line arrangement but more generally for hyperplane arrangements.
\begin{conj*}[Terao]
 The freeness of a hyperplane arrangement depends only on its  combinatorics.
\end{conj*}
In other words an arrangement with the same combinatorics of a free arrangement is free, too.

 Let us consider a free arrangement $\mathcal{A}_0=\{h_1,\ldots, h_n\}$ with exponents $(a,b)$ ($a\le b$)  and let us denote by $Z_0$ its dual set of points.
 We assume that Terao's conjecture is not true i.e,  that there exists
a non free arrangement $\mathcal{A}=\{l_1,\ldots, l_n\}$ with the same combinatorics  of $\mathcal{A}_0$.
 Let us add $b-a$ points $\{M_1,\ldots, M_{b-a}\}$ in general position to $Z_0$
  in order to form $\Gamma_0$ and to th dual set $Z$ of$\mathcal{A}$ to form $\Gamma$.
Then the length of both sets of points is $2b+1$. On the general line $l$ we have
$$\mathcal{D}_0(Z_0)\otimes \sO_l=\sO_l(-a)\oplus \sO_l(-b),$$
when, since $Z$ is not free,  we have a less balanced decomposition for $\mathcal{D}_0(Z)$ (this affirmation is  proved  in \cite{EF}):
$$\mathcal{D}_0(Z)\otimes \sO_l=\sO_l(s-a)\oplus \sO_l(-s-b),\,\, s\ge 1.$$
It implies that
$$\mathrm{h}^0(\cI_Z\otimes \cI_{l^{\vee}}^{a-1}(a))\neq 0, \mathrm{h}^0(\cI_{Z_0}\otimes \cI_{l^{\vee}}^{a-1}(a))= 0\,\,
\mathrm{and}\,\,\mathrm{h}^0(\cI_{Z_0}\otimes \cI_{l^{\vee}}^{a}(a+1))\neq 0.$$
Then adding $b-a$ lines passing through $l^{\vee}$ and the $b-a$ added points we obtain
 $\mathrm{h}^0(\cI_{\Gamma}\otimes \cI_{l^{\vee}}^{b-1}(b))\neq 0$, $\mathrm{h}^0(\cI_{Z_0}\otimes \cI_{l^{\vee}}^{b-1}(b))= 0$ and
$\mathrm{h}^0(\cI_{Z_0}\otimes \cI_{l^{\vee}}^{b}(b+1))\neq 0.$
The bundle
$\mathcal{D}_0(\Gamma_0)$ is balanced with splitting
$\sO_l(-b)\oplus \sO_l(-b)$ and  $$\mathcal{D}_0(\Gamma)\otimes \sO_l=\sO_l(1-b)\oplus \sO_l(-1-b).$$
Then  $\mathcal{D}_0(\Gamma_0)$ is semistable and $\mathcal{D}_0(\Gamma)$ is unstable.
In other words  the ideal
$$ ( l_1^b, \ldots, l_{a+b+1}^b, (M_1^{\vee})^b, \ldots, (M_{b-a}^{\vee})^b)$$ fails the SLP at the range $2$ and degree $b-2$ when
$$ ( d_1^b, \ldots, d_{a+b+1}^b, (M_1^{\vee})^b, \ldots, (M_{b-a}^{\vee})^b)$$
has the SLP at the range $2$ and degree $b-2$.

\noindent The following conjecture written in terms of SLP  is equivalent to  Terao's conjecture on $\p^2$.
\begin{conj*}
Let $Z_0=\{d_1^{\vee}, \ldots, d_{2b+1}^{\vee}\}$ a set of points of length $2b+1$ in $\p^{2\vee}$ such that the ideal
$ I= (d_1^b, \ldots, d_{2b+1}^b)$ has the SLP at the range $2$ and degree $b-2$. Assume that $Z=\{l_1^{\vee}, \ldots, l_{2b+1}^{\vee}\}$ has the same combinatorics of $Z_0$. Then
$ J= (l_1^b, \ldots, l_{2b+1}^b)$ has the SLP at the range $2$ and degree $b-2$.
\end{conj*}

\address{
   Roberta Di Gennaro and Giovanna Ilardi\\
   Universit\`a di Napoli ``Federico II", Dipartimento di Matematica e Applicazioni \lq\lq R. Caccioppoli\rq\rq, Complesso Universitario Monte Sant'Angelo, Ed. 5, Via Cinthia 80126 Napoli,\\
   Italy\\
   \email{digennar@unina.it \\
   giovanna.ilardi@unina.it}}
   \\
\address{
 Jean Vall\`es\\
Universit\'e de Pau et des Pays de l'Adour, Laboratoire de Math\'ematique- B\^at. IPRA, Avenue de l'Universit\'e, 64000 Pau \\France
\email{jean.valles@univ-pau.fr}}


\begin{thebibliography}{A-Wykno}


\bibitem{A}  {\bibname Valdimir I. Arnold}, The cohomology ring of colored braid group.  {\it Mat. Zametki}, 5(2): 227--231, 1969.
    \bibitem{BCGI} {\bibname Alessandra Bernardi, Maria Virginia Catalisano, Alessandro Gimigliano and  Monica Id\'a}, Osculating varieties of Veronese varieties and their higher secant varieties.
 {\it Canad. J. Math.} 59, no. 3, 488--502  (2007).
\bibitem{BK} {\bibname Holger Brenner and Almar Kaid},  Syzygy bundles on $\p^2$ and the Weak Lefschetz Property. {\it  Illinois J. Math.}, 51:1299--1308, 2007.
\bibitem{Cartier} {\bibname Pierre Cartier}, Les arrangements d'hyperplans: un chapitre de g\'eom\'etrie combinatoire.
Bourbaki Seminar, Vol. 1980/81, p. 1--22, Lecture Notes in Math., 901, Springer, Berlin-New York, 1981.
\bibitem{DDI}  {\bibname Pietro De Poi, Roberta Di Gennaro and Giovanna Ilardi}, On varieties with higher osculating defect, {\it arXiv:1204.4399}. To appear on Revista Matem\'{a}tica Iberoamericana, 29 (4) 2013.
\bibitem{D} {\bibname Igor Dolgachev}, {\it Classical Algebraic Geometry: a modern view}. To be published by Cambridge Univ. press.
\bibitem{EF}  {\bibname Georges Elencwajg and Otto Forster}, Bounding cohomology groups of vector bundles on {${\p}_{n}$}, {\it Math. Ann.}, 246(3):251--270, 1979/80.
\bibitem{EI} {\bibname Jacques  Emsalem and Anthony Iarrobino}, Inverse System of a Symbolic Power, I, {\it J. of Algebra} 174:1080--1090, 1995.
\bibitem{FI} {\bibname Davide Franco and Giovanna Ilardi},  On a theorem of Togliatti. {\it  Int. Math. J.} 2(4):379--397, 2002.
\bibitem{FMV} {\bibname Daniele Faenzi, Daniel Matei and Jean Vall\`es},  Hyperplane arrangements of Torelli type.  {\it arXiv:1011.4611}, November 2010. To appear on Compositio Math.
\bibitem{H} {\bibname Joe Harris},  {\it Algebraic geometry, a first course}.  volume 133 of {\it Graduate Texts in Math.} Springer Verlag, 1992.
\bibitem{HMNW} {\bibname Tadahito Harima, Juan C. Migliore, Uwe Nagel and Junzo Watanabe}, The weak and Strong Lefschetz Properties for artinian K-Algebras.  {\it J. Algebra}, 262:99--126, 2003.
\bibitem{HSS} {\bibname Brian Harbourne, Henry K. Schenk and Alexandra Seleleanu}, Inverse systems, Gelfand-Tsetlin patterns and the weak Lefschetz property.
{\it Journal of the London Mathematical Society}  84:712--730, 2011.
\bibitem{I2} {\bibname Giovanna Ilardi}, Togliatti systems. {\it Osaka J. Math.},  43(1):1--12, 2006.
\bibitem{LM} {\bibname Antonio Lanteri and Raquel  Mallavibarena}, Osculatory behavior and second dual varieties of Del Pezzo surfaces. {\it Adv. in Geom.}, 2(4):345--363, 2002.
\bibitem{MMN} {\bibname Juan C. Migliore, Rosa  Mir\'o-Roig and Uwe Nagel}, On the weak lefschetz property for powers of linear forms. {\it Algebra and Number theory}, 4:487--526, 2012.
\bibitem{MMO} {\bibname Emilia Mezzetti, Rosa  Mir\'o-Roig and Giorgio Ottaviani},   Laplace Equations and the Weak Lefschetz Property, Canad. J. Math. 65(2013), 634-654.
\bibitem{MN} {\bibname Juan C. Migliore and Uwe Nagel},  A tour of the weak and strong Lefschetz properties. {\it arXiv:1109.5718}, September 2011. To
appear on Journal of Commutative Algebra.
\bibitem{OSS} {\bibname Christian Okonek, Michael Schneider and Heinz Spindler}, {\it Vector bundles on complex projective spaces}.
 Progress in Mathematics, vol. 3,  Birkh\"auser, Boston, Mass., 1980.
\bibitem{OT} {\bibname Peter Orlik and Hiroaki Terao}, {\it  Arrangement of hyperplanes},  volume 300 of {\it Grundlerhen der Mathematischen Wissenschaften [Fundamental
Principles of Mathematical Sciences].} Springer-Verlag, Berlin,  1992.
\bibitem{P} {\bibname David Perkinson}, Inflections of toric varieties.   {\it Mich. Math. J.}, 48:483--516 , 2000.
\bibitem{S} {\bibname Kyoji Saito}, Theory of logarithmic differential forms and logarithmic vector fields. {\it J. Fac. Sci. Univ. Tokyo Sect. IA Math.}, 27(2):265--291, 1980.
\bibitem{Sc} {\bibname Henry K. Schenk},  Elementary modifications and line configurations in $\p^2$.
{\it  Comment. Math. Helv.} 78(3):447--462, 2003.
\bibitem{St} {\bibname Richard P. Stanley},  Weyl groups, the hard Lefschetz theorem, and the Sperner property.
{\it  Siam J. Alg. Disc. Meth.} 1(2):168--184, 1980.
\bibitem{Te} {\bibname Alessandro Terracini}, Sulle $V_k$ che rappresentano pi\`u di $\frac{1}{2} k (k - 1)$ equazioni di Laplace linearmente indipendenti.
{\it Rend. Circ. Mat. Palermo}, 33:176--186, 1912.
\bibitem{To} {\bibname Eugenio Togliatti}, Alcune osservazioni sulle superficie razionali che rappresentano equazioni di Laplace. {\it  Ann. Mat. Pura Appl.} 25(4):325--339, 1946.
\bibitem{V1} {\bibname Jean Vall\`es}, Vari\'et\'es de type Togliatti. {\it C.R.A.S.} 343(6):411--414, 2006.
\bibitem{V2} {\bibname Jean Vall\`es}, Fibr\'es logarithmiques sur le plan projectif. {\it Ann. Fac. Sci. Toulouse Math.} 16(2):385--395, 2007.

\end{thebibliography}
\end{document}